\documentclass[11pt,a4paper]{article}
\usepackage{amssymb,latexsym}
\usepackage{bm, amsmath, amssymb, amsthm}

\def\Hb{H^\bot}
\def\Ht{H^\top}

\def\<{\langle}
\def\>{\rangle}
\def\eps{\epsilon}
\def\RR{\mathbb{R}}

\def\tr{\operatorname{Tr\,}}

\def\Div{\operatorname{div}}

\def\Ric{\operatorname{Ric}}
\def\vol{\operatorname{vol}}

\def\eq{\hspace*{-.5mm}&=&\hspace*{-.5mm}}

\def\minus{\hspace*{-.5mm}&-&\hspace*{-.5mm}}

\newtheorem{theorem}{Theorem}
\newtheorem{corollary}{Corollary}
\newtheorem{lemma}{Lemma}

\newtheorem{definition}{Definition}
\newtheorem{remark}{Remark}

\title{Integral formulas for a metric-affine manifold \\ with two complementary orthogonal distributions}

\author{
       Vladimir Rovenski\footnote{Mathematical Dept., University of Haifa, Mount Carmel, Haifa, 31905, Israel.
       \newline
       e-mail: {\tt vrovenski@univ.haifa.ac.il}} }

\begin{document}

\date{}

\maketitle

\begin{abstract}
We obtain integral formulas for a metric-affine space equipped with two complementary ortho\-gonal distributions.
The integrand depends on the Ricci and mixed scalar curvatures and invariants of the second fundamental forms and integrability tensors
of the distributions. The formulas under some conditions yield splitting of manifolds (including submersions and twisted products)
and provide geometrical obstructions for existence of distributions and foliations (or compact leaves of them).
\end{abstract}




\section*{Introduction}

Integral formulas provide obstructions for existence of distributions and foliations (or compact leaves of them)
with given geometric properties and have applications in different areas of geometry and analysis, see survey in \cite{r25,rw-m}.
Distributions on manifolds appear in various situations, e.g. as fields of tangent planes of foliations
or kernels of differential forms \cite{bf}.
The~first known integral formula by G.\,Reeb \cite{Ree} for a codimension-1 foliated closed Riemannian manifold $(M,g)$ tells~us~that the~total mean curvature $H$ of the leaves is zero (thus, either $H\equiv0$ or $H(x)H(y)<0$ for some points $x,y\in M$). The proof is based on the divergence theorem and the identity $\Div N = (\dim M)\,h_{sc}$,
where $N$ is a unit normal to the leaves and $h_{sc}$ their scalar second fundamental form.
The formula poses a generalization (which is a consequence of Green's theorem applied to $N$)
to the case of second order mean curvature $\sigma_2$:
\begin{equation}\label{E-sigma2}
 \int_M (2\,\sigma_2-\Ric_{N,N})\,{\rm d}\vol_g=0.
\end{equation}
Moreover, \eqref{E-sigma2} admits a leaf-wise counterpart for a closed leaf $M'$ with induced metric $g'$.
Both formulas have many applications: e.g., \eqref{E-sigma2} implies nonexis\-tence
of umbilical foliations on a closed manifold of negative curvature.
Later on \cite{aw2010} extended \eqref{E-sigma2} to infinite series of formulas with higher order mean curvatures $\sigma_k\ (k\ge2)$.
 In~further generalizations for complementary orthogonal distributions of any dimension \cite{lu-w,r8,wa1},
 the integrand depends on second fundamental forms and integrability tensors of the distributions, their
derivatives and curvature invariants, e.g. the Ricci curvature and the mixed scalar curvature.

No attempt has been made to develop integral formulas for general metric-affine spaces equipped with distributions.
The~{Metric-Affine~Geo\-metry} founded by E.\,Cartan
uses an asymmetric connection $\bar\nabla$ with torsion (instead of the Levi-Civita connection $\nabla$);
it generalizes Riemannian Geometry and appears in such context as homogeneous and almost Hermitian manifolds.
The~important distinguished cases are:
{Riemann-Cartan manifolds}, where {metric connections}, i.e., $\bar\nabla g =0$, are used, see e.g. \cite{gps},
and {statistical manifolds}, see \cite{bc16,ch12}, where the torsion is zero and the tensor $\bar\nabla g$ is~symmetric in all its entries.
The theory of affine hypersurfaces in $\RR^{n+1}$ is a natural source of statistical manifolds, see \cite{lsz15}.
Riemann-Cartan spaces are central in gauge theory of gravity, where the torsion is represented by the spin tensor of matter.

In this paper, we obtain integral formulas for general metric-affine spaces (equipped with distributions) and for two distinguished~classes.
 The formulas naturally generalize results for Riemannian case \cite{lu-w,step1,step2,wa1} with the Ricci and mixed scalar curvatures,
 under some conditions they yield splitting of  ambient manifolds (including submersions and twisted products)
and provide geometrical obstructions for existence of distributions and foliations or compact leaves of them.

\section{Preliminaries}\label{subsec:mixedEH}

Let $M^{n+p}$ be a connected smooth manifold with a pseudo-Riemannian metric $g$ of index $q$ and
complementary orthogonal non-degene\-rate distributions ${\cal D}^\top$ and ${\cal D}^\bot$
(subbundles of the tangent bundle $TM$ of ranks
$\dim_{\,\RR}{\cal D}^\top_x=n$ and $\dim_{\,\RR}{\cal D}^\bot_x=p$ for every $x\in M$).
A~distribution ${\cal D}^\top$ is {non-degenerate}, if
$g_x\ (x\in M)$ is a {non-degenerate bilinear form} on ${\cal D}^\top_x\subset T_x M$
for every $x\in M$; in this case, ${\cal D}^\bot$ is also non-degenerate.
When~$q=0$, $g$ is a Riemannian metric (resp. a Lorentz metric~when $q=1$), see~\cite{bf}.
Let $^\top$ and $^\perp$ denote $g$-orthogonal projections onto ${\cal D}^\top$ and ${\cal D}^\bot$, respectively;
for any $X\in\mathfrak{X}_M$ we write $X=X^\top+X^\bot$.
  We~will define several tensors for one of distributions
(say, ${\cal D}^\top$; similar tensors for the second distribu\-tion can be defined using $\,^\bot$ notation).
 The~following convention is adopted for the range~of~indices:
\[
 a,b\ldots{\in}\{1\ldots n\},\quad
 i,j\ldots{\in}\{1\ldots p\}.
\]
One may show that the local adapted orthonormal frame $\{E_a,\,{\cal E}_{i}\}$,
where $\{E_a\}\subset{{\cal D}^\top}$, and $\eps_i=g({\cal E}_{i},{\cal E}_{i})\in\{-1,1\}$, $\eps_a=g(E_a,E_a)\in\{-1,1\}$,
always exists on $M$.

Let
 $T^\top, h^\top:{\cal D}^\top\times {\cal D}^\top\to{\cal D}^\bot$
be the integrability tensor and the second fundamental form of~${\cal D}^\top$,
\begin{eqnarray*}
  T^\top(X,Y):=(1/2)\,[X,\,Y]^\perp,\quad h^\top(X,Y): = (1/2)\,(\nabla_X Y+\nabla_Y X)^\perp.
\end{eqnarray*}
The mean curvature vector of ${\cal D}^\top$ is
 $\Ht =\sum\nolimits_a\eps_a h^\top(E_a,E_a)$.
 We call ${\cal D}^\top$ \textit{umbilical}, \textit{harmonic}, or \textit{totally geodesic}, if
 $h^\top=\frac1n \Ht\,g^\top$, $\Ht =0$, or $h^\top=0$, resp.
The~Weingarten operator $A^\top$ (of ${\cal D}^\top$)
 and the operator ${T}^{\top\sharp}$ are defined by
\[
 g(A^\top_Z X,Y)= g(h^\top(X^\top,Y^\top),Z^\bot),\quad
 g({T}^{\top\sharp}_Z X,Y)=g(T^\top(X^\top,Y^\top),Z^\bot).
\]
We will use the following convention for various tensors:
 ${T}^{\top\sharp}_i := {T}^{\top\sharp}_{{\cal E}_i},\ A^\top_i := A^\top_{ {\cal E}_i }$, ${\cal T}_i
 ={\cal T}_{{\cal E}_i}$ etc.

One of the simplest curvature invariants of a pseudo-Riemannian manifold $(M,g)$ endowed with
two complementary orthogonal distributions $({\cal D}^\top,{\cal D}^\bot)$
is the \textit{mixed scalar curvature}, i.e., an averaged sectional curvature of planes that non-trivially intersect the distribution ${\cal D}^\top$ and its orthogonal complement ${\cal D}^\bot$, see~\cite{wa1}:
\begin{equation}\label{eq-wal2}
 {\rm S}_{\,\rm mix} =\sum\nolimits_{\,a,i}\eps_a \eps_i\,g(R_{a,i}E_a,\, {\cal E}_{i}).
\end{equation}
Here $R_{X,Y}=[\nabla_Y,\nabla_X] + \nabla_{[X,Y]}$ is the curvature tensor of $\nabla$.
The~following formula \cite{wa1} has found many applications:
\begin{equation}\label{E-PW}
  \Div(\Hb + \Ht) = {\rm S}_{\,\rm mix} + \<h^\bot,h^\bot\> -g(\Hb,\Hb) -\<T^\bot,T^\bot\>
  +\<h^\top,h^\top\> - g(\Ht,\Ht) -\<T^\top,T^\top\> ,
\end{equation}
see survey in \cite{r25,rw-m}.  We use inner products of tensors, e.g.
\begin{eqnarray*}
 \<h^\top,h^\top\>\eq\sum\nolimits_{\,i,j}\eps_i\eps_j\,g(h^\top({\cal E}_i,{\cal E}_j),h^\top({\cal E}_i,{\cal E}_j)),\quad
 \<T^\top,T^\top\>=\sum\nolimits_{\,i,j}\eps_i\eps_j\,g(T^\top({\cal E}_i,{\cal E}_j),T^\top({\cal E}_i,{\cal E}_j)).
\end{eqnarray*}
Let~$\mathfrak{X}_M$ (resp., $\mathfrak{X}_{{\cal D}^\top}$) be the module over $C^\infty(M)$
of all vector fields on $M$ (resp. on~${\cal D}^\top$).
 A \textit{metric-affine space} is a manifold $M$ endowed with a metric $g$ of certain signature
and a linear connection $\bar\nabla:\mathfrak{X}_M\times\mathfrak{X}_M\to\mathfrak{X}_M$ on $TM$ that is
\[
 \bar\nabla_{f X_1+X_2} Y =f\bar\nabla_{X_1} Y+\bar\nabla_{X_2} Y,\quad
 \bar\nabla_X(fY +Z) =X(f)Y+f\bar\nabla_{X}Y +\bar\nabla_{X}Z.
\]
The Levi-Civita connection~$\nabla$ is a unique torsion free connection on $(M,g)$ preserving $g$.
It can be taken as a center of affine space of all connections on $M$.
The difference ${\mathcal T}:=\bar\nabla -\nabla$ is called the \textit{contorsion tensor}.
 Define the (1,2)-tensors ${\cal T}^*$ and $\widehat{\cal T}$ by
\[
 g({\cal T}^*_X Y,Z) = g( {\cal T}_X Z , Y),\quad
 \widehat{\cal T}_X Y = {\cal T}_Y X ,
 \quad X,Y,Z\in\mathfrak{X}_M.
\]
Remark that generally $(\widehat{\cal T})^* \ne \widehat{{\cal T}^*}$.
Indeed,
\[
 g((\widehat{\cal T})^*_X Y,Z)=g(\widehat{\cal T}_X Z,Y)=g({\cal T}_Z X,Y),\quad
 g((\widehat{{\cal T}^*})_XY, Z) = g({\cal T}^*_YX, Z) =g({\cal T}_YZ, X).
\]

A connection $\bar\nabla = \nabla + {\cal T}$
is \textit{metric compatible} if $\bar\nabla g=0$; in~this case,
\begin{equation*}
 {\cal T}^* = -{\cal T}.
\end{equation*}
If $\bar\nabla$ is torsionless and tensor $\bar\nabla g$ is symmetric in all its entries then
$\bar\nabla$ is called a \emph{statistical connection} in literature.
In~this case, the contorsion tensor has the following symmetries:
\begin{equation*}
 \widehat{\cal T} = {\cal T} , \quad {\cal T}^* = {\cal T}.
\end{equation*}

Comparing the curvature tensor
$\bar R_{X,Y}=[\bar\nabla_Y,\bar\nabla_X]+\bar\nabla_{[X,Y]}$ of $\bar\nabla$ with $R$, we~find
\begin{equation}\label{E-RC-2}
 \bar R_{X,Y} - R_{X,Y} = (\nabla_Y {\cal T})_X -(\nabla_X {\cal T})_Y +[{\cal T}_Y,\,{\cal T}_X],\quad X,Y\in\mathfrak{X}_M.
\end{equation}
Define two \textit{mean curvature type vectors of ${\cal T}$} by
 $\Ht_{\cal T} := \sum\nolimits_a\eps_a {\cal T}_a E_a$
 and
 $\Hb_{\cal T} := \sum\nolimits_i\eps_i {\cal T}_i {\cal E}_i$.

\begin{remark}\rm
One can examine the extrinsic geometry also  in terms of ${\bar\nabla}$.
For example,
\begin{eqnarray*}
 \bar h^\top(X,Y) = h^\top(X,Y) +\frac{1}{2}\,({\cal T}_{X} Y +{\cal T}_{Y} X)^{\perp}, \quad X,Y\in{\cal D}^\top,
\end{eqnarray*}
is the second fundamental form of ${\cal D}^\top$ w.r.t. ${\bar\nabla}$,
and its mean curvature vector is
 ${\bar H}^\top = \Ht +(\Ht_{\cal T})^\bot$.
\end{remark}

\begin{definition}\rm
 The following function on a metric-affine manifold $(M,g,\bar\nabla)$ endowed with
 two complementary orthogonal distributions $({\cal D}^\top,{\cal D}^\bot)$:
\begin{equation}\label{eq-wal3}
 \bar{\rm S}_{\,\rm mix} = ({1}/{2})\sum\nolimits_{a,i} \eps_a\eps_i\big( g( {\bar R}_{a,i} E_a, {\cal E}_i)
 + g( {\bar R}_{i,a}\,{\cal E}_i, E_a ) \big)
\end{equation}
is called the \textit{mixed scalar curvature} w.r.t.~$\bar\nabla$, see~\eqref{eq-wal2} for the Riemannian case.
\end{definition}

Definition \eqref{eq-wal3} does not depend on the order of distributions and the choice of a frame.
Thus, by~\eqref{E-RC-2} and \eqref{eq-wal2},
\begin{eqnarray}\label{E-new1}
\nonumber
 \bar{\rm S}_{\,\rm mix} \eq \frac{1}{2}\sum\nolimits_{a,i} \eps_a \eps_i
 \big( g( (\nabla_i {\cal T})_a E_a , {\cal E}_i ) - g( (\nabla_a {\cal T})_i E_a, {\cal E}_i )
 + g( (\nabla_a {\cal T})_i\, {\cal E}_i, E_a )\\
 \minus g( (\nabla_i {\cal T})_a\, {\cal E}_i , E_a )
  + g( [{\cal T}_i,\, {\cal T}_a] E_a , {\cal E}_i ) + g( [{\cal T}_a,\, {\cal T}_i]\, {\cal E}_i , E_a ) \big)
  + {\rm S}_{\,\rm mix}.
\end{eqnarray}
 The Divergence Theorem for a vector field $\xi$ on  $(M,g)$ states that
\begin{equation}\label{E-Div-Th}
 \int_M (\Div \xi)\,{\rm d}\vol_g=0,
\end{equation}
 when either $\xi$ has compact support or $M$ is closed.
The~${\cal D}^\bot$-\textit{divergence} of $\xi$ is defined by
$\Div^\perp\xi=\sum\nolimits_{i} \eps_i\,g(\nabla_{i}\,\xi, {\cal E}_i)$,
and for $\xi\in\mathfrak{X}_{{\cal D}^\bot}$ we have
 $\,{\Div}^\bot \xi = \Div \xi+g(\xi,\,\Ht)$.

\section{Integral formulas}
\label{sec:IF}

The main idea of proving integral formulas (as in the case of formulas discussed in the introduction) is to calculate the divergence of a
vector field and use the Stokes Theorem.
In \cite{lu-w}, this approach was applied to vector fields $Z_k=(A^\top_{\Ht})^k\Hb + (A^\bot_{\Hb})^k\Ht$,
on a Riemannian space, namely, for $k=0,1$.
We add to $Z_0=\Hb + \Ht$ (in Section~\ref{sec:2.1})
and $Z_1=A^\top_{\Ht}\Hb + A^\bot_{\Hb}\Ht$ (in Section~\ref{sec:2.2}) certain vector fields on a metric-affine space,
compute their divergence and find integral formulas.

We also work with a closed manifold $M$ equipped with vector fields, distributions, etc.
defined on the complement to the ``set of singularities" $\Sigma$, which is a (possibly empty) union of
closed submanifolds of variable codimension $\ge2$.
The singular case is important since many manifolds admit no smooth (e.g. codimension-one) distributions,
while they admit such distributions and tensors outside some $\Sigma$.
Example of such distributions provide ``open book decompositions" on manifolds, see discussion in \cite{lu-w}.

\begin{lemma}[see Lemma~2 in \cite{lu-w}]\label{L-LuW-2}
Let $\Sigma_1$, ${\rm codim}\, \Sigma_1\ge2$, be a closed submanifold of a Riemannian manifold $(M,g)$,
and $\xi$ a vector field on $M\setminus\Sigma_1$ such that $\|\xi\|_g\in L^2(M,g)$.
Then \eqref{E-Div-Th} holds.
\end{lemma}

\subsection{Integral formulas with $\bar{\rm S}_{\,\rm mix}$}
\label{sec:2.1}

The divergence of $Z_0$, see \eqref{E-PW}, was calculated in \cite{wa1} and the integral formula
\begin{equation}\label{E-PW-IF}
 \int_{M} \big\{\,{\rm S}_{\,\rm mix} -\|T^\top\|^2 -\|T^\bot\|^2 +\|h^\top\|^2 +\|h^\bot\|^2 -\|\Ht\|^2 -\|\Hb\|^2\big\}\,{\rm d}\vol_g = 0
\end{equation}
was obtained for a closed Riemannian manifold $M$.
 We will denote by $\<B,C\>_{|V}$ the inner product of tensors $B,C$ restricted on the subbundle
$V=({\cal D}^\top\times{\cal D}^\bot)\cup({\cal D}^\bot\times{\cal D}^\top)$.

\begin{lemma}\label{L-QQ-first}
We have
\begin{eqnarray}\label{E-Q1Q2-gen}
\nonumber
 && \Div\big( (\Ht_{\cal T} - \Ht_{{\cal T}^*})^\bot + (\Hb_{\cal T} - \Hb_{{\cal T}^*})^\top \big) =
 2\,(\bar{\rm S}_{\,\rm mix} - {\rm S}_{\,\rm mix}) \\
\nonumber
 && -\,g( \Ht_{\cal T},\,\Hb_{{\cal T}^*}) - g(\Hb_{\cal T},\,\Ht_{{\cal T}^*})
 -g(\Ht_{\cal T} -\Hb_{\cal T} + \Hb_{{\cal T}^*} - \Ht_{{\cal T}^*},\, \Ht - \Hb) \\
 && -\,\< {\cal T} -{\cal T}^* +\widehat{\cal T} -\widehat{{\cal T}^*},\,
 A^\bot\!-T^{\bot\sharp} \!+A^\top\!-T^{\top\sharp}\>
 +\<{\cal T}^*,\,\widehat{\cal T}\>_{\,|\,V}.
\end{eqnarray}
For statistical manifolds, \eqref{E-Q1Q2-gen} reads
\begin{equation*}
 \bar{\rm S}_{\,\rm mix} - {\rm S}_{\,\rm mix} - g( \Ht_{\cal T},\,\Hb_{\cal T}) + (1/2)\,\<{\cal T},\,{\cal T}\>_{\,|\,V} = 0 .
\end{equation*}
For Riemann-Cartan manifolds, \eqref{E-Q1Q2-gen} reads
\begin{eqnarray*}
 \Div\big( (\Ht_{\cal T})^\bot + (\Hb_{\cal T})^\top \big) =
 \bar{\rm S}_{\,\rm mix} - {\rm S}_{\,\rm mix}
 +g( \Ht_{\cal T},\,\Hb_{\cal T}) - g( \Ht_{\cal T} - \Hb_{\cal T},\, \Ht - \Hb)\\
 -\< {\cal T}+\widehat{\cal T},\,A^\bot-T^{\bot\sharp} +A^\top-T^{\top\sharp}\>
 - (1/2)\,\<{\cal T},\,\widehat{\cal T}\>_{\,|\,V}.
\end{eqnarray*}
\end{lemma}

\begin{proof}
Using \eqref{E-new1} we get $\bar{\rm S}_{\,\rm mix} - {\rm S}_{\,\rm mix} = \frac12\,(Q_1 + Q_2)$, where
\begin{eqnarray*}
 Q_1 \eq \sum\nolimits_{a,i}\eps_a\eps_i\,[\,g((\nabla_i{\cal T})_a E_a, {\cal E}_i) - g((\nabla_a{\cal T})_i E_a, {\cal E}_i)
 + g([{\cal T}_i,\, {\cal T}_a] E_a , {\cal E}_i )\,],\\
 Q_2\eq\sum\nolimits_{a,i}\eps_a\eps_i\,[\,g((\nabla_a{\cal T})_i{\cal E}_i, E_a) -g((\nabla_i{\cal T})_a{\cal E}_i, E_a)
 +g([{\cal T}_a,\, {\cal T}_i]\, {\cal E}_i , E_a )\,].
\end{eqnarray*}
Let $(\nabla_i E_a)^\top=0$ and $(\nabla_a {\cal E}_i)^\bot=0$ at a point $x\in M$, thus
\[
 -\nabla_i E_a=\sum\nolimits_{j}\eps_j\,g((A^\bot_a+T^{\bot\sharp}_a){\mathcal E}_i,{\mathcal E}_j){\mathcal E}_j,\quad
 -\nabla_a\,{\mathcal E}_i=\sum\nolimits_{b}\eps_b\,g((A^\top_i+T^{\top\sharp}_i)E_a,E_b) E_b.
\]
 We calculate at $x$:
\begin{eqnarray*}
  g((\nabla_i{\cal T})_a E_a, {\cal E}_i) \eq
  g({\cal T}_i E_a +{\cal T}_a {\cal E}_i,\, (A^\bot_a-T^{\bot\sharp}_a){\cal E}_i) +\Div^\bot(\Ht_{\cal T}),\\
 g((\nabla_a {\cal T})_i E_a, {\cal E}_i) \eq
 g({\cal T}^*_a {\cal E}_i +{\cal T}^*_i E_a,\, (A^\top_i-T^{\top\sharp}_i)E_a) +\Div^\top(\Hb_{{\cal T}^*}),\\
  g((\nabla_a {\cal T})_i{\cal E}_i, E_a) \eq
 g({\cal T}_a {\cal E}_i +{\cal T}_i E_a,\, (A^\top_i-T^{\top\sharp}_i)E_a) +\Div^\top(\Hb_{\cal T}),\\
%
 g((\nabla_i {\cal T})_a {\cal E}_i, E_a) \eq
 g({\cal T}^*_i E_a +{\cal T}^*_a {\cal E}_i,\, (A^\bot_a-T^{\bot\sharp}_a){\cal E}_i) +\Div^\bot(\Ht_{{\cal T}^*}),\\
  g([{\cal T}_i,\, {\cal T}_a] E_a, {\cal E}_i ) \eq  g( \Ht_{\cal T},\,\Hb_{{\cal T}^*})
 -
 g({\cal T}_i E_a, {\cal T}^*_a{\cal E}_i ),\\
  g([{\cal T}_a,\, {\cal T}_i]\, {\cal E}_i, E_a ) \eq g( \Hb_{\cal T},\,\Ht_{{\cal T}^*})
  -
  g({\cal T}^*_i E_a, {\cal T}_a{\cal E}_i ),
\end{eqnarray*}
(omitting $\sum\nolimits_{a,i}\eps_a\eps_i$) and find
\begin{eqnarray}\label{E-Q1Q2-gen1}
\nonumber
 Q_1 \eq \Div^\bot(\Ht_{\cal T}) - \Div^\top(\Hb_{{\cal T}^*})
 +\sum\nolimits_{a,i}\eps_a\eps_i\,\big[ g({\cal T}_i E_a +{\cal T}_a {\cal E}_i,\, (A^\bot_a-T^{\bot\sharp}_a){\cal E}_i)\\
\nonumber
 \minus g({\cal T}^*_a {\cal E}_i +{\cal T}^*_i E_a,\, (A^\top_i-T^{\top\sharp}_i)E_a)
  -g({\cal T}_i E_a, {\cal T}^*_a{\cal E}_i ) \big] +g( \Ht_{\cal T},\, \Hb_{{\cal T}^*}),\\
\nonumber
 Q_2 \eq \Div^\top(\Hb_{\cal T}) -\Div^\bot(\Ht_{{\cal T}^*})
 + \sum\nolimits_{a,i}\eps_a\eps_i\,\big[g({\cal T}_a {\cal E}_i +{\cal T}_i E_a,\, (A^\top_i-T^{\top\sharp}_i)E_a)\\
 \minus g({\cal T}^*_i E_a +{\cal T}^*_a {\cal E}_i,\, (A^\bot_a-T^{\bot\sharp}_a){\cal E}_i)
  -g( {\cal T}_a{\cal E}_i, {\cal T}^*_i E_a )\big] + g( \Hb_{\cal T},\, \Ht_{{\cal T}^*}).
\end{eqnarray}
From \eqref{E-Q1Q2-gen1}, using equalities
\begin{eqnarray*}
 &&\hskip-5mm \Div^\bot(\Ht_{\cal T}) = \Div( (\Ht_{\cal T})^\bot ) + g(\Ht_{\cal T}, \Ht-\Hb),\quad
  \Div^\top(\Hb_{{\cal T}^*}) = \Div( (\Hb_{{\cal T}^*})^\top) - g(\Hb_{{\cal T}^*}, \Ht-\Hb),\\
 &&\hskip-5mm \Div^\bot(\Ht_{{\cal T}^*}) = \Div( (\Ht_{{\cal T}^*})^\bot) + g((\Ht_{{\cal T}^*}, \Ht{-}\Hb),\quad
  \Div^\top(\Hb_{\cal T}) = \Div( (\Hb_{\cal T})^\top) - g(\Hb_{\cal T}, \Ht{-}\Hb)
\end{eqnarray*}
we obtain
\begin{eqnarray*}
 &&  \Div\big( (\Ht_{\cal T} - \Ht_{{\cal T}^*})^\bot + (\Hb_{\cal T} - \Hb_{{\cal T}^*})^\top \big)
 = 2\,(\bar{\rm S}_{\,\rm mix} - {\rm S}_{\,\rm mix}) \\
 && -\,g( \Ht_{\cal T},\,\Hb_{{\cal T}^*}) - g(\Hb_{\cal T},\,\Ht_{{\cal T}^*})
 - g(\Ht_{\cal T} -\Hb_{\cal T} + \Hb_{{\cal T}^*} - \Ht_{{\cal T}^*},\, \Ht - \Hb) \\
 && -\sum\nolimits_{a,i}\eps_a\eps_i\,\big[\,g( ({\cal T}_i -{\cal T}^*_i) E_a
 +({\cal T}_a -{\cal T}^*_a) {\cal E}_i,\,(A^\bot_a-T^{\bot\sharp}_a){\cal E}_i +(A^\top_i-T^{\top\sharp}_i)E_a)\\
  && -\,g( {\cal T}_a\,{\cal E}_i,\,{\cal T}^*_i E_a) - g({\cal T}^*_a\,{\cal E}_i,\,{\cal T}_i\,E_a)\,\big],
\end{eqnarray*}
and hence, \eqref{E-Q1Q2-gen}.
\end{proof}

In the next theorem we generalize \eqref{E-PW-IF}.

\begin{theorem}
Let $(M,g,\bar\nabla)$ be a closed metric-affine space and ${\cal D}^\top$ a distribution defined
on the complement to the ``set of singularities" $\Sigma$, see Section~\ref{subsec:mixedEH}.
If $\|\xi\|_g\in L^2(M,g)$, where
$\xi=\Hb +\Ht\!+\frac12\,(\Ht_{\cal T} -\Ht_{{\cal T}^*})^\bot +\frac12\,(\Hb_{\cal T} -\Hb_{{\cal T}^*})^\top$,
then the following integral formula holds:
\begin{eqnarray*}
 && \hskip-6mm
 \int_{M}\!\big\{\bar{\rm S}_{\rm mix} {-}\<T^\top,T^\top\> -\<T^\bot,T^\bot\> +\<h^\top,h^\top\> +\<h^\bot,h^\bot\>
 -g(\Ht,\Ht) -g(\Hb,\Hb) \\
 && -\,\frac12\,\big[\,g( \Ht_{\cal T},\, \Hb_{{\cal T}^*}) + g( \Hb_{\cal T},\, \Ht_{{\cal T}^*})
 + g( \Ht_{\cal T} - \Hb_{\cal T} + \Hb_{{\cal T}^*} - \Ht_{{\cal T}^*},\,\Ht - \Hb)\,\big]\\
 && -\,\frac12\,\< {\cal T} -{\cal T}^* +\widehat{\cal T} -\widehat{{\cal T}^*},\,
 A^\bot\!-T^{\bot\sharp} \!+A^\top\!-T^{\top\sharp}\>
 +\frac12\,\<{\cal T}^*,\,\widehat{\cal T}\>_{\,|\,V} \big\}\,{\rm d}\vol_g = 0 .
\end{eqnarray*}
\end{theorem}

\begin{proof}
By \eqref{E-PW} and Lemma~\ref{L-QQ-first}, we have on $M\setminus\Sigma$:
\begin{eqnarray}\label{E-div-IF1}
\nonumber
 &&\hskip-2mm \Div\xi = \bar{\rm S}_{\rm mix} {-}\<T^\top,T^\top\> -\<T^\bot,T^\bot\> +\<h^\top,h^\top\> +\<h^\bot,h^\bot\> -g(\Ht,\Ht) -g(\Hb,\Hb) \\
\nonumber
 && -\frac12\,\big[\,g( \Ht_{\cal T},\, \Hb_{{\cal T}^*}) + g( \Hb_{\cal T},\,\Ht_{{\cal T}^*})
 + g( \Ht_{\cal T} -\Hb_{\cal T} + \Hb_{{\cal T}^*} - \Ht_{{\cal T}^*},\, \Ht - \Hb)\,\big] \\
 && -\,\frac12\,\< {\cal T} -{\cal T}^* +\widehat{\cal T} -\widehat{{\cal T}^*},\,
 A^\bot\!-T^{\bot\sharp} \!+A^\top\!-T^{\top\sharp}\>
 +\frac12\,\<{\cal T}^*,\,\widehat{\cal T}\>_{\,|\,V}.
\end{eqnarray}
Thus, the claim follows from \eqref{E-div-IF1} and Lemma~\ref{L-LuW-2}.
\end{proof}

\begin{corollary}
Let $(M,g,\bar\nabla)$ be a closed statistical manifold and ${\cal D}^\top$ a distribution defined
on the complement to the ``set of singularities" $\Sigma$. If $\|\Hb +\Ht\|_g\in L^2(M,g)$ then
\begin{eqnarray*}
 &&\hskip-10mm \int_{M} \big\{\,\bar{\rm S}_{\,\rm mix} -\<T^\top,T^\top\> -\<T^\bot,T^\bot\> +\<h^\top,h^\top\> +\<h^\bot,h^\bot\> - g(\Ht,\Ht)\\
 \minus g(\Hb,\Hb) - g( \Ht_{\cal T},\, \Hb_{\cal T})
 +(1/2)\,\<{\cal T},\,{\cal T}\>_{\,|\,V}\big\}\,{\rm d}\vol_g = 0 .
\end{eqnarray*}
\end{corollary}

\begin{corollary}
Let $(M,g,\bar\nabla)$ be a closed Riemann-Cartan manifold and ${\cal D}^\top$ a distribution defined
on the complement to the ``set of singularities" $\Sigma$. If $\|\bar H^\bot +\bar H^\top\|_g\in L^2(M,g)$ then
\begin{eqnarray*}
 &&\hskip-5mm \int_{M} \big\{\,\bar{\rm S}_{\,\rm mix} -\<T^\top,T^\top\> -\<T^\bot,T^\bot\> +\<h^\top,h^\top\> +\<h^\bot,h^\bot\> -g(\Ht,\Ht) -g(\Hb,\Hb) +g( \Ht_{\cal T},\, \Hb_{\cal T})\\
 && -\,g( \Ht_{\cal T} - \Hb_{\cal T},\, \Ht - \Hb)
 -\< {\cal T} -\widehat{\cal T},\, A^\bot-T^{\bot\sharp} +A^\top-T^{\top\sharp}\>
 -(1/2)\,\<{\cal T},\,\widehat{\cal T}\>_{\,|\,V}\big\}\,{\rm d}\vol_g = 0 .
\end{eqnarray*}
\end{corollary}

\begin{definition}\rm
We say that $(M',g')$ is a \textit{leaf of a distribution} ${\cal D}$ on $(M,g)$ if $M'$ is a submanifold of $M$ with induced metric $g'$ and
$T_xM'={\cal D}_x$ for any $x\in M'$. A leaf $(M',g')$ of ${\cal D}$ is said to be
\textit{umbilical}, \textit{harmonic}, or \textit{totally geodesic}, if
 $h^\top=\frac1n\,\Ht g^\top,\ \Ht =0$, or $h^\top=0$, respectively, along~$M'$.
\end{definition}

The next theorem generalizes result in \cite{wa1}.

\begin{theorem}
Let a distribution ${\cal D}^\top$ on a metric-affine space $(M,g,\bar\nabla)$ has a compact leaf $ (M',g')$ with condition
\begin{equation}\label{E-mixed-ai}
 2\Ht = (\Ht_{{\cal T}^*}-\Ht_{\cal T} )^\bot
\end{equation}
on its neighborhood. Then  the following integral formula along the leaf holds:
\begin{eqnarray}\label{E-IF-leaf0}
\nonumber
 &&\hskip-3mm \int_{M'} \!\big\{\bar{\rm S}_{\,\rm mix} - \<T^\bot,T^\bot\> +\<h^\bot,h^\bot\> +\<h^\top,h^\top\>
+\frac12\,\big[ g(\Ht_{\cal T} - \Ht_{{\cal T}^*}, \Hb) + g(\Hb_{\cal T} - \Hb_{{\cal T}^*}, \Ht)\\
 &&\hskip-10mm -g(\Ht_{\cal T}, \Hb_{{\cal T}^*}) -g(\Hb_{\cal T}, \Ht_{{\cal T}^*})
 -\< {\cal T}-{\cal T}^*\!+\widehat{\cal T}-\widehat{{\cal T}^*}, A^\bot \!-T^{\bot\sharp} \!+A^\top\>
 +\<{\cal T}^*,\widehat{\cal T}\>_{\,|\,V}  \big] \big\}\,{\rm d}\vol_{g'} =0 .
\end{eqnarray}
\end{theorem}

\begin{proof}
Using $T^\top=0$, \eqref{E-PW}, Lemma~\ref{L-QQ-first} and equalities
\begin{eqnarray*}
 && \Div^\bot(\Hb) = -g(\Hb,\Hb),\quad \Div^\top(\Ht) = -g(\Ht,\Ht),\\
 && \Div^\bot( (\Hb_{\cal T} - \Hb_{{\cal T}^*})^\top) = -g(\Hb_{\cal T}-\Hb_{{\cal T}^*}, \Hb), \\
 && \Div^\top( (\Ht_{\cal T} - \Ht_{{\cal T}^*})^\bot) = -g(\Ht_{\cal T}-\Ht_{{\cal T}^*}, \Ht),
\end{eqnarray*}
we have
\begin{eqnarray}\label{E-div-IF2}
\nonumber
 && \Div^\top\big(\Hb +\frac12\,(\Hb_{\cal T}-\Hb_{{\cal T}^*})^\top\big)
 +\Div^\bot \big(\Ht +\frac12\,(\Ht_{\cal T}-\Ht_{{\cal T}^*})^\bot\big) = \\
\nonumber
 && = \bar{\rm S}_{\,\rm mix} +\<h^\bot,h^\bot\> +\<h^\top,h^\top\> - \<T^\bot,T^\bot\>
 +\frac12\,\big[ g(\Ht_{\cal T} - \Ht_{{\cal T}^*}, \Hb) + g(\Hb_{\cal T} - \Hb_{{\cal T}^*}, \Ht)\\
 && -g(\Ht_{\cal T}, \Hb_{{\cal T}^*}) -g(\Hb_{\cal T}, \Ht_{{\cal T}^*})
 -\< {\cal T}-{\cal T}^*+\widehat{\cal T}-\widehat{{\cal T}^*},\,A^\bot \!-T^{\bot\sharp} \!+A^\top\>
 +\<\,{\cal T}^*,\,\widehat{\cal T}\>_{\,|\,V}
 \big] .
\end{eqnarray}
By conditions \eqref{E-mixed-ai}, the $\Div^\bot$-term in \eqref{E-div-IF2} vanishes along $M'$.
Thus, \eqref{E-IF-leaf0} follows from the Divergence theorem
for $\xi={H}^\bot + \frac12\,(\Hb_{\cal T}-\Hb_{{\cal T}^*})^\top$ on $M'$.
\end{proof}

\begin{corollary}
Let a distribution ${\cal D}^\top$ on a statistical manifold $(M,g,\bar\nabla)$ admits a compact leaf $ (M',g')$
with condition $\Ht=0$ on its neighborhood. Then  the following integral formula holds:
\begin{equation*}
 \int_{M'} \!\big\{\bar{\rm S}_{\rm mix} \!- \<T^\bot,T^\bot\> +\<h^\bot,h^\bot\> +\<h^\top,h^\top\>
 - g( \Ht_{\cal T},\, \Hb_{\cal T})\!-\<{\cal T},{\cal T}\>_{|\,V} \big\}{\rm d}\vol_{g'} =0 .
\end{equation*}
\end{corollary}

\begin{corollary}
Let a distribution ${\cal D}^\top$ on a Riemann-Cartan manifold $(M,g,\bar\nabla)$ admits a compact leaf $ (M',g')$
with condition ${\bar H}^\top = 0$ on its neighborhood. Then the following integral formula holds:
\begin{eqnarray*}
 && \int_{M'} \!\big\{\bar{\rm S}_{\,\rm mix} - \<T^\bot,T^\bot\> +\<h^\bot,h^\bot\> +\<h^\top,h^\top\>
 + g( \Ht_{\cal T},\, \Hb_{\cal T}) + g( \Ht_{\cal T},\,\Hb) \\
 &&\ \ +\, g( \Hb_{\cal T},\,\Ht) -\<{\cal T}+\widehat{\cal T}, A^\bot \!-T^{\bot\sharp} \!+ A^\top\>
 -(1/2)\,\<{\cal T},\,\widehat{\cal T}\>_{\,|\,V}\big\}\,{\rm d}\vol_{g'} =0 .
\end{eqnarray*}
\end{corollary}

\begin{remark}\rm
For ${\cal D}^\bot$ spanned by a unit vector field $N$, set
 $\overline{\Ric}_{N,N} = \sum\nolimits_a\,\eps_a\,g(\bar R_{N,E_a}N, E_a)$.
Note that generally $g(\bar R_{N,E_a}N, E_a)\ne g(\bar R_{E_a,N}E_a, N)$.
Let $\eps_N=1$. Similarly to Lemma~\ref{L-QQ-first}, we have
\begin{eqnarray*}
 && \Div( ({\mathcal T}_NN)^\top -(\Ht_{{\mathcal T}^*})^\bot\,) = \overline{\Ric}_{N,N} - {\Ric}_{N,N} + Q,\\
 && Q = g(\Ht_{{\cal T}^*} +{\cal T}_N N,\, H^\bot -(\tr A^\top)N)
 -\<(\widehat{\cal T})^*_N+{\cal T}^*_N, A^\top_N + T^{\top\sharp}_N\>_{\,|{\cal D}^\top} \\
 &&\quad +\,g((\widehat{\cal T})^*_N N +{\cal T}_N N,\Hb) -g(\Ht_{{\cal T}^*}, {\cal T}_NN)
 +\<\widehat{\cal T}_N,{\cal T}^*_N\>_{\,|{\cal D}^\top}.
\end{eqnarray*}
The above yield the integral formula, which for ${\cal T}=0$ reduces to~\eqref{E-sigma2} with $\sigma_2=\tr (A^\top_N)^2$:
\begin{eqnarray*}
 \int_M (2\,\sigma_2-\overline\Ric_{N,N} -Q)\,{\rm d}\vol_g=0.
\end{eqnarray*}
\end{remark}

\subsection{Integral formula with the Ricci curvature $\overline{\Ric}_{\Ht,\Hb}$}
\label{sec:2.2}

The divergence of $Z_1$ for a Riemannian manifold endowed with orthogonal complementary distributions ${\cal D}^\top$ and ${\cal D}^\bot$ has been calculated in \cite{lu-w}:
\begin{eqnarray}\label{EW-AHH-a}
 && \Div(A^\top_{\Ht}\Hb + A^\bot_{\Hb}\Ht) = \Ric_{\Ht,\Hb} + \,Q_1,
\end{eqnarray}
where
\begin{eqnarray}\label{EW-AHH}
\nonumber
 && \hskip-6mm Q_1 = g(\Ht, \nabla_{\Hb}\Ht) +g(\Hb, \nabla_{\Ht}\Hb) \\
\nonumber
 && +\,g(\tr^\bot_g (\nabla_{\centerdot}\, T^\top)({\bm\cdot}\,,\Hb),\Ht)
 +g(\tr^\top_g (\nabla_{\centerdot}\, T^\bot)({\bm\cdot}\,,\Ht),\Hb)\\
\nonumber
 && +\,\<A^\top_{\Ht}, \nabla_{\centerdot}\,\Hb\> +\<A^\bot_{\Hb}, \nabla_{\centerdot}\, \Ht\>
 -g(A^\top_{\Hb}\Ht,\Ht) -g(A^\bot_{\Ht}\Hb,\Hb)\\
\nonumber
 && +\,2\sum\nolimits_a \eps_a\big[\,g(A^\top_{(\nabla_a \Ht)^\bot}\Hb, E_a)
    +g(\nabla_{T^{\top}(\Hb,\,E_a)}E_a, \Ht)\,\big] \\
 && +\,2\sum\nolimits_i \eps_i\big[\,g(A^\bot_{(\nabla_i\Hb)^\top} \Ht, {\cal E}_i)
    +g(\nabla_{T^{\bot}(\Ht,\,{\cal E}_i)}{\cal E}_i, \Hb)\,\big].
\end{eqnarray}
Thus, on a closed manifold $(M,g)$ one has the integral formula, see \cite[Theorem~1]{lu-w},
\begin{equation}\label{EW-Ric-1}
 \int_M (\Ric_{\Ht,\Hb} + \,Q_1)\,{\rm d}\vol_g = 0.
\end{equation}
If the distributions are umbilical, integrable and have constant mean curvature then \eqref{EW-AHH} reads
\begin{equation*}
 Q_1 = -(1/n+1/p)\,g(\Ht,\Ht)\,g(\Hb,\Hb) .
\end{equation*}

\begin{lemma} For the metric-affine case we have
\begin{equation}\label{E-RicHH-a}
 \Div({\cal T}_{\Hb} \Ht +{\cal T}_{\Ht} \Hb) = -(\overline{\Ric}_{\Ht,\Hb} - \Ric_{\Ht,\Hb}) + Q_2,
\end{equation}
 where
 $\overline{\Ric}_{\Ht,\Hb} = {\rm Sym}\big(\sum\nolimits_a\,\eps_a\,g(\bar R_{\Hb,\,E_a}\Ht, E_a)
 +\sum\nolimits_i\,\eps_i\,g(\bar R_{\Ht,\,{\cal E}_i}\Hb, {\cal E}_i)\big)$
 and
\begin{eqnarray}\label{E-RicHH-b}
\nonumber
 && \hskip-6mm Q_2 = {\rm Sym}\Big(
 g(\bar\nabla_{\Hb}\Ht, \Ht_{{\cal T}^*}) +g(\bar\nabla_{\Ht}\Hb, \Hb_{{\cal T}^*})
 -g({\cal T}_{\Hb} \Ht, \Hb) - g({\cal T}_{\Ht} \Hb, \Ht) \\
\nonumber
 && -\sum\nolimits_a \eps_a\,g\big(\nabla_{\Hb}({\cal T}_a \Ht)
 -{\cal T}_{(h^\top+T^\top)(\Hb,E_a) +\nabla_a\Hb} \Ht + {\cal T}_{\Hb}(\bar\nabla_a\Ht),\, E_a\big) \\
 && -\sum\nolimits_i \eps_i\,g\big(\nabla_{\Ht}({\cal T}_i \Hb) - {\cal T}_{(h^\bot+T^\bot)(\Ht,{\cal E}_i)
 +\nabla_i\Ht} \Hb + {\cal T}_{\Ht}(\bar\nabla_i \Hb),\, {\cal E}_i \big)\Big) .
\end{eqnarray}
\end{lemma}

\begin{proof}
Using \eqref{E-RC-2}, we calculate
\begin{eqnarray}\label{E-RicHH2}
\nonumber
 && \sum\nolimits_a \eps_a\,g(\bar R_{\Hb\,E_a}\Ht - R_{\Hb,\,E_a}\Ht, E_a) = \\
\nonumber
 && ={\rm Sym}\big(\Div^\top({\cal T}_{\Hb} \Ht) +g(\nabla_{\Hb}\Ht, \Ht_{{\cal T}^*}) +g({\cal T}_{\Hb} \Ht, \Ht_{{\cal T}^*}) \\
\nonumber
 && -\sum\nolimits_a \eps_a\big[ g(\nabla_{\Hb}({\cal T}_a \Ht),E_a) -g({\cal T}_{(h^\top+T^\top)(\Hb,E_a)} \Ht,E_a) \\
 && -\,g({\cal T}_{\nabla_a\Hb} \Ht,E_a) +g({\cal T}_{\Hb}(\nabla_a \Ht), E_a)
 +g({\cal T}_{\Hb}({\cal T}_a \Ht),E_a) \big]\big) .
\end{eqnarray}
Summing \eqref{E-RicHH2} with similar formula for
$\sum\nolimits_i \eps_i\,g(\bar R_{{\cal E}_i,\Ht}\Hb - R_{{\cal E}_i,\Ht}\Hb, {\cal E}_i)$
and using equalities
\begin{eqnarray*}
  \Div^\bot({\cal T}_{\Ht}\Hb) = \Div({\cal T}_{\Ht}\Hb) - g(\Ht,{\cal T}_{\Ht}\Hb),\\
  \Div^\top({\cal T}_{\Hb}\Ht) = \Div({\cal T}_{\Hb}\Ht) - g(\Hb,{\cal T}_{\Hb}\Ht),
\end{eqnarray*}
yield \eqref{E-RicHH-a}--\eqref{E-RicHH-b}.
\end{proof}

\begin{theorem}\label{T-IF-05}
Let $(M,g,\bar\nabla)$ be a closed metric-affine space and ${\cal D}^\top$ a distribution defined
on the complement to the ``set of singularities" $\Sigma$.
If $\|\xi\|_g\in L^2(M,g)$, where $\xi=A^\top_{\Ht}\Hb \!+ A^\bot_{\Hb}\Ht \!+ {\cal T}_{\Ht}\Hb +{\cal T}_{\Hb}\Ht$, then
\begin{equation}\label{E-IF-01b}
 \int_{M} \big\{\,\overline{\Ric}_{\Ht,\Hb} + Q_1 + Q_2\big\}\,{\rm d}\vol_g = 0.
\end{equation}
\end{theorem}

\begin{proof}
From \eqref{EW-AHH-a} and \eqref{E-RicHH-a} we obtain
\begin{equation*}
 \Div(A^\top_{\Ht}\Hb + A^\bot_{\Hb}\Ht +{\cal T}_{\Hb} \Ht +{\cal T}_{\Ht} \Hb) = -\overline{\Ric}_{\Ht,\Hb} + Q_1 + Q_2.
\end{equation*}
Applying Lemma~\ref{L-LuW-2} to \eqref{EW-AHH}, \eqref{E-RicHH-a} and \eqref{E-RicHH-b}, we obtain~\eqref{E-IF-01b}.
\end{proof}

For ${\cal T}=0$, we have $Q_2=0$, and \eqref{E-IF-01b} reduces to \eqref{EW-Ric-1}.
One may get a number of formulas from \eqref{E-IF-01b}.
The~next one generalizes Proposition~3 in \cite{lu-w}.

Recall that ${\cal D}^\top$ has constant mean curvature whenever its mean curvature vector $\Ht$ obeys
$\nabla^\bot\Ht=0$, where $\nabla^\bot$ is the connection in ${\cal D}^\bot$ induced by the Levi-Civita connection on $M$.

\begin{corollary}
Let in conditions of Theorem~\ref{T-IF-05}, distributions ${\cal D}^\top$ and ${\cal D}^\bot$ be umbilical, integrable and have constant mean curvature. Then  \eqref{E-IF-01b}~reads
\begin{eqnarray*}
 && \int_{M} \big\{ \overline{\Ric}_{\Ht,\Hb} -(\frac1n+\frac1p)\,g(\Ht,\Ht)\,g(\Hb,\Hb)
 +{\rm Sym}\big( g(\bar\nabla_{\Ht}\Hb, \Hb_{{\cal T}^*})\\
 && +\,g(\bar\nabla_{\Hb}\Ht, \Ht_{{\cal T}^*}) - g(\Hb,\ {\cal T}_{\Hb -\Ht}\,\Ht) - g(\Ht,\ {\cal T}_{\Ht -\Hb}\,\Hb) \\
 && -\sum\nolimits_a \eps_a\,g( \bar\nabla_{\Hb}({\cal T}_a \Ht) - {\cal T}_{\nabla_a\Hb}\Ht
  + {\cal T}_{\Hb}(\nabla_a \Ht),\ E_a) \\
 && -\sum\nolimits_i \eps_i\,g( \bar\nabla_{\Ht}({\cal T}_i \Hb)
 - {\cal T}_{\nabla_i\Ht} \Hb + {\cal T}_{\Ht}(\nabla_i \Hb),\ {\cal E}_i)\big)\big\}\,{\rm d}\vol_g = 0 .
\end{eqnarray*}
\end{corollary}

\section{Splitting results}

In this section, we consider distributions $({\cal D}^\top,{\cal D}^\bot)$ on
a metric-affine manifold
$(M,g,\bar\nabla)$, satisfying some geometrical
conditions, and prove non-existence and splitting results, which follow from integral formulas of Section~\ref{sec:IF}.
In the sequel, we assume that $g>0$.
 We omit similar results for (co)dimension one distributions and foliations
and consequences for harmonic and Riemannian submersions, which follow from results below.

We say that $(M,g)$ endowed with a distribution ${\cal D}^\top$ \textit{splits} if $M$
is locally isometric to a product with canonical foliations tangent to ${\cal D}^\top$ and
${\cal D}^\bot$. Remark that if a simply connected manifold splits then it is the direct product.

The next conditions are introduced to simplify the presentation of results:
\begin{eqnarray}\label{E-cond4}
 && {\cal T}_X Y=0={\cal T}_Y X,\quad {\cal T}^*_X Y=0={\cal T}^*_Y X\quad (X\in{\cal D}^\top,\ Y\in{\cal D}^\bot), \\
\label{E-cond5}
 && g(\Ht_{\cal T}- \Ht_{{\cal T}^*}\,,\, H^\bot) = 0.
\end{eqnarray}
For example, \eqref{E-cond4} provides vanishing of last two lines in \eqref{E-div-IF1}.

 Some recent extensions of \eqref{E-Div-Th} to non-compact case are discussed in~\cite{step1}.
Applying S.T.Yau version of Stokes' theorem on a complete open Riemannian manifold $(M,g)$ yields the following.

\begin{lemma}[see Proposition~1 in \cite{csc2010}]\label{L-Div-1}
 Let $(M,g)$ be a complete open Riemannian manifold endowed with a vector field $\xi$
 such that $\Div\xi\ge0$. If the norm $\|\xi\|_g\in L^1(M,g)$ then $\Div \xi\equiv0$.
\end{lemma}

\subsection{Harmonic distributions}

Note that condition \eqref{E-mixed-ai} on $M$ when ${\cal T}=0$ reduces to $H^\top=0$, i.e., ${\cal D}^\top$ is harmonic.

Next theorem generalizes Theorem~5 in \cite{step1}.

\begin{theorem}
Let ${\cal D}^\top$ and ${\cal D}^\bot$ be
complementary ortho\-gonal integrable distributions with \eqref{E-mixed-ai}, \eqref{E-cond4} and \eqref{E-cond5} on a complete open metric-affine space $(M,g,\bar\nabla)$.
Suppose that the leaves $(M',g')$ of ${\cal D}^\top$ obey condition $\|\xi_{\,|M'}\|_{g'}\in L^1(M',g')$, where
$\xi={H}^\bot +\frac12\,(\Hb_{\cal T} - \Hb_{{\cal T}^*})^\top$. If  $\bar{\rm S}_{\,\rm mix}\ge0$ then $M$ splits.
\end{theorem}

\begin{proof}
By conditions, we have
\begin{equation}\label{E-th4}
 \Div^\top \xi = \bar{\rm S}_{\,\rm mix} +\<h^\bot,h^\bot\> +\<h^\top,h^\top\>.
\end{equation}
Applying Lemma~\ref{L-Div-1} to each leaf (a complete open manifold),
and since $\bar{\rm S}_{\,\rm mix}\ge0$, we get $\Div\xi=0$.
Thus, $h^\top=0=h^\bot$. By de Rham decomposition theorem, $(M,g)$ splits.
\end{proof}

Note that condition $\|\,\xi_{|M'}\|_{g'}{\in} L^1(M',g')$ is satisfied on any compact leaf $(M',g')$ of~${\cal D}^\top$.

Next two results generalize Theorem~2 and Corollary~4 in \cite{wa1}.

\begin{corollary}\label{T-w1}
Let $(M,g,\bar\nabla)$ be a metric-affine space endowed with a distribution ${\cal D}^\top$
with integrable normal bundle and conditions \eqref{E-mixed-ai}, \eqref{E-cond4} and \eqref{E-cond5}.
Then~${\cal D}^\top$ has no complete open leaves $(M',g')$ with the properties
$\bar{\rm S}_{\,{\rm mix}\,|M'}>0$ and $\|\xi_{\,|M'}\|_{g'}\in L^1(M',g')$, where
$\xi={H}^\bot +\frac12\,(\Hb_{\cal T} - \Hb_{{\cal T}^*})^\top$.
In particular, there are no compact leaves with $\bar{\rm S}_{\,{\rm mix}\,|M'}>0$.
\end{corollary}

\begin{proof}
Let $(M',g')$ be a complete open leaf obeying the conditions.
By \eqref{E-div-IF2}, we have \eqref{E-th4} on $M'$. Applying Lemma~\ref{L-Div-1}
to the leaf, and since $\bar{\rm S}_{\,\rm mix}>0$, we get $\Div\xi=0$.
The above yields $h^\top=0=h^\bot$ and $\bar{\rm S}_{\,\rm mix}=0$ -- a contradiction.
\end{proof}

\begin{corollary}
A codimension one distribution ${\cal D}^\top$ of a metric-affine space $(M,g,\bar\nabla)$
with the Ricci curvature $\overline\Ric>0$ and
the properties \eqref{E-mixed-ai}, \eqref{E-cond4} and \eqref{E-cond5} has no compact leaves.
\end{corollary}

\begin{proof}
For a codimension one ${\cal D}^\top$, we have $T^\top=0$ and $\eps_N\,\overline\Ric_{N,N}=\bar{\rm S}_{\,\rm mix}$,
where $N$ is a unit normal to the leaves. Hence, the claim follows from Corollary~\ref{T-w1}.
\end{proof}

\begin{theorem}
Let $(M,g,\bar\nabla)$ be a complete open (or closed) metric-affine space
endowed with complementary orthogonal harmonic foliations. Suppose that conditions \eqref{E-cond4},
$\|\xi\|_g\in L^1(M,g)$, where $\xi=\frac12\,(\Ht_{\cal T}-\Ht_{{\cal T}^*})^\bot +\frac12\,(\Hb_{\cal T}-\Hb_{{\cal T}^*})^\top$,
and
\begin{equation*}
 g( \Ht_{\cal T},\, \Hb_{{\cal T}^*}) + g( \Hb_{\cal T},\,\Ht_{{\cal T}^*}) =0
\end{equation*}
are satisfied. If $\bar{\rm S}_{\,\rm mix}\ge0$ then $M$~splits.
\end{theorem}

\begin{proof}
Under conditions, from \eqref{E-div-IF1} we~get
\begin{equation*}
 \Div \xi = \bar{\rm S}_{\,\rm mix} +\<h^\top,h^\top\> +\<h^\bot,h^\bot\> .
\end{equation*}
By Lemma~\ref{L-Div-1} and since $\bar{\rm S}_{\,\rm mix}\ge0$, we get $\Div\xi=0$.
Thus, $h^\top=0=h^\bot$. Hence, by de Rham decomposition theorem, $(M,g)$ splits.
\end{proof}

\subsection{Umbilical distributions}

\begin{theorem}\label{T-w2}
Let $(M,g,\bar\nabla)$ be a metric-affine space endowed with
two complementary orthogo\-nal distributions $({\cal D}^\top,{\cal D}^\bot)$
with conditions \eqref{E-mixed-ai}, \eqref{E-cond4} and \eqref{E-cond5}.
Then~${\cal D}^\top$ has no complete open umbilical leaves $(M',g')$ with the properties
$\bar{\rm S}_{\,{\rm mix}\,|M'}<0$ and
 $\|\xi_{\,|M'}\|_{g'}\in L^1(M',g')$, where $\xi={H}^\bot +\frac12\,( \Hb_{\cal T} - \Hb_{{\cal T}^*})^\top$.
In particular, there are no compact umbilical leaves $(M',g')$ with $\bar{\rm S}_{\,{\rm mix}\,|M'}<0$.
\end{theorem}

\begin{proof}
Let $(M',g')$ be a complete open umbilical leaf obeying the conditions.
From \eqref{E-div-IF2} we~get
\begin{equation}\label{E-umb-T6}
 \Div^\top \xi = \bar{\rm S}_{\,\rm mix}
 -\<T^\bot, T^\bot\>
 -\frac{p-1}{p}\,g(\Hb,\Hb) -\frac{n-1}{n}\,g(\Ht,\Ht)
\end{equation}
on $M'$.
Thus, applying Lemma~\ref{L-Div-1} to the leaf, and since $\bar{\rm S}_{\,{\rm mix}\,|M'}<0$, we get $\Div^\top\xi=0$.
By~the~above, $\Ht=0=\Hb$, $T^\bot=0$ and $\bar{\rm S}_{\,\rm mix}=0$ -- a contradiction.
\end{proof}

\begin{corollary}
A codimension one distribution ${\cal D}^\top$ on a metric-affine space $(M,g,\bar\nabla)$ with the Ricci curvature $\overline\Ric<0$
and conditions \eqref{E-mixed-ai}, \eqref{E-cond4} and \eqref{E-cond5} has no compact umbilical leaves.
\end{corollary}

A submersion $f:(M,g)\to (\tilde M,\tilde g)$ is \textit{conformal} if $f_*$ restricted to $(\ker f_*)^\bot$
is conformal map,~\cite{step1}.

\begin{theorem}[For ${\cal T}=0$, see Corollary~5 in \cite{step1}]
Let $(M,g,\bar\nabla)$ be a complete open metric-affine space,
$f:(M,g)\to (\tilde M,\tilde g)$ a conformal submersion with umbilical fibres
and conditions \eqref{E-mixed-ai}, \eqref{E-cond4} and~\eqref{E-cond5}.
If $\bar{\rm S}_{\,{\rm mix}}\le0$
and $\|\xi_{|M'}\|_{g'}{\in} L^1(M',g')$, where $\xi={H}^\bot +\frac12(\Hb_{\cal T} - \Hb_{{\cal T}^*})^\top$,
on any fibre $(M'g')$ then $(\ker f_*)^\bot$ is integrable and $M$~splits.
\end{theorem}

\begin{proof} Set ${\cal D}^\top=\ker f$. Under conditions, we have \eqref{E-umb-T6}.
Applying Lemma~\ref{L-Div-1} to every fibre (a complete open manifold), and since $\bar{\rm S}_{\,\rm mix}\le0$,
we get $\Div\xi=0$. The above yields vanishing of $\Ht$, $\Hb$ and $T^\bot$. By de Rham decomposition theorem, $(M,g)$ splits.
\end{proof}

\begin{theorem}[For ${\cal T}=0$, see Theorem~4 in \cite{step1}]
\label{C-Step3} Let $(M,g,\bar\nabla)$ be a complete open (or closed) metric-affine space
endowed with complementary orthogonal umbilical distributions
${\cal D}^\top$ and ${\cal D}^\bot$ defined on the complement to the ``set of singularities" $\Sigma$.
If conditions \eqref{E-cond4},
\begin{equation}\label{E-HHT-b}
 \Ht_{\cal T} = 0 = \Hb_{\cal T}, \quad \Ht_{{\cal T}^*} = 0 = \Hb_{{\cal T}^*}
\end{equation}
and $\|\xi\|_g\in L^1(M,g)$, where $\xi=\Hb +\Ht$, are satisfies and $\bar{\rm S}_{\,\rm mix}\le0$ then $M$~splits.
\end{theorem}

\begin{proof}
 Under conditions, from \eqref{E-div-IF1} we~get
\begin{equation}\label{E-umb-C5}
 \Div \xi = \bar{\rm S}_{\,\rm mix} -\<T^\top,T^\top\>
 -\<T^\bot,T^\bot\> -\frac{p-1}{p}\,g(\Hb,\Hb) -\frac{n-1}{n}\,g(\Ht,\Ht).
\end{equation}
From \eqref{E-umb-C5} and Lemma~\ref{L-Div-1} and since $\bar{\rm S}_{\,\rm mix}\le0$, we get $\Div\xi=0$.
The above yields vanishing of $T^\top,T^\bot,\Ht,\Hb$. Hence, by de Rham decomposition theorem, $(M,g)$ splits.
\end{proof}

Umbilical integrable distributions appear on double-twisted products, see \cite{pr}.

\begin{definition}\rm
A \textit{doubly-twisted product} $B\times_{(v,u)} F$ of metric-affine spaces $(B,g_B,{\cal T}_B)$
and $(F, g_F,{\cal T}_F)$ is a manifold $M=B\times F$ with metric $g = g^\top + g^\perp$ and contorsion tensor
${\cal T} = {\cal T}^\top + {\cal T}^\bot$,~where
\begin{eqnarray*}
 g^\top(X,Y) \eq v^2 g_B(X^\top,Y^\top),\quad g^\bot(X,Y)=u^2 g_F(X^\bot,Y^\bot),\\
 ({\cal T}^\top)_XY \eq u^2({\cal T}_B)_{X^\top}Y^\top,\quad ({\cal T}^\bot)_XY=v^2({\cal T}_F)_{X^\bot}Y^\bot,
\end{eqnarray*}
and the warping functions $u,v\in C^\infty(M)$ are positive.
Indeed, ${\cal T}^* = {\cal T}^{*\top} + {\cal T}^{*\bot}$, where
\[
 ({\cal T}^{*\top})_XY=u^2({\cal T}^*_B)_{X^\top}Y^\top,\quad
 ({\cal T}^{*\bot})_XY=v^2({\cal T}^*_F)_{X^\bot}Y^\bot.
\]
\end{definition}

Let ${\cal D}^\top$ be tangent to the \textit{fibers} $\{x\}\times F$ and ${\cal D}^\bot$
tangent to the \textit{leaves} $B\times\{y\}$. The second fundamental forms
and the mean curvature vectors of $B\times_{(v,u)} F$ are given by, see \cite{pr},
\begin{eqnarray*}
 && h^\bot = -\nabla^\top(\log u)\,g^\bot,\quad h^\top=-\nabla^\bot(\log v)\,g^\top,\\
 && \Hb=-n\,\nabla^\top(\log u),\quad \Ht=-p\,\nabla^\bot(\log v).
\end{eqnarray*}
Thus, the {leaves} and the {fibers} of $B\times_{(v,u)} F$ are umbilical w.r.t. $\bar\nabla$ and $\nabla$.
 Conditions \eqref{E-cond4} are obviously satisfied for $B\times_{(v,u)} F$.
Next corollaries extend results in~\cite{step2}.

\begin{corollary}[of Theorem~\ref{T-w2}] Let \eqref{E-mixed-ai}, \eqref{E-cond5} and
 $\|\xi_{\,|M'}\|_{g'}\in L^1(M',g')$, $\xi={H}^\bot +\frac12\,( \Hb_{\cal T} - \Hb_{{\cal T}^*})^\top$,
 are satisfied  along the fibres of $B\times_{(v,u)} F$, where $(F,g_F)$ is complete open (or closed).
 If $\bar{\rm S}_{\,\rm mix}\le0$ then $M$ is the direct product.
\end{corollary}

\begin{proof} Under conditions, from \eqref{E-umb-T6} we~get
\begin{equation*}
 \Div^\top \xi = \bar{\rm S}_{\,\rm mix} -\frac{p-1}{p}\,g(\Hb,\Hb) -\frac{n-1}{n}\,g(\Ht,\Ht) .
\end{equation*}
Applying Lemma~\ref{L-Div-1} to each fibre (a complete manifold), and since $\bar{\rm S}_{\,\rm mix}\le0$,
we get $\Div\xi=0$. Hence, $\bar{\rm S}_{\,\rm mix}=0$ and $\Ht=0=\Hb$, i.e., $\nabla^\top u=0=\nabla^\bot v$.
By the above, ${\rm S}_{\,\rm mix}=0$; thus, $u$ and $v$ are constant.
 By de Rham decomposition theorem, $(M,g)$ splits with the factors $(B,c_1\cdot g_B)$ and $(F,c_2\cdot g_F)$ for some $c_1,c_2>0$.
\end{proof}

\begin{corollary}[of Theorem~\ref{C-Step3}]
 Let $M=B\times_{(v,u)} F$ be complete open (or closed) and conditions \eqref{E-cond5}, \eqref{E-HHT-b} and $\|\xi\|_{g}\in L^1(M,g)$,
 $\xi=\Hb + \Ht$ are satisfied.
 If $\bar{\rm S}_{\,\rm mix}\le0$ then $M$ is the direct product.
\end{corollary}

\begin{proof} Set ${\cal D}^\top=\pi_*(TF)$. Under conditions, we get, see \eqref{E-umb-C5},
\begin{equation*}
 \Div \xi = \bar{\rm S}_{\,\rm mix} -\frac{p-1}{p}\,g(\Hb,\Hb) -\frac{n-1}{n}\,g(\Ht,\Ht) .
\end{equation*}
Applying Lemma~\ref{L-Div-1} to $M$, and since $\bar{\rm S}_{\,\rm mix}\le0$,
we get $\Div\xi=0$. Hence,  $\bar{\rm S}_{\,\rm mix}=0$ and $\Ht=0=\Hb$, i.e., $\nabla^\top u=0=\nabla^\bot v$.
By the above, ${\rm S}_{\,\rm mix}=0$; thus, $u$ and $v$ are constant.
By de Rham decomposition theorem, $(M,g)$ splits with the factors $(B,\,c_1\cdot g_B)$ and $(F,\,c_2\cdot g_F)$ for some positive $c_1,c_2\in\RR$.
\end{proof}

\textbf{Acknowledgments}.
The author wishes to thank
Prof. Sergey Stepanov (Moscow) for useful conversations during the XIX Geometrical Seminar, Zlatibor, 2016,
and Prof.~Pawel Walczak (Lodz) for comments concerning this work.

\end{document}